\documentclass{amsart}
\usepackage{amsfonts}
\usepackage{esint}
 
\newtheorem{thm}{Theorem}[section]
\newtheorem{lema}[thm]{Lemma}
\theoremstyle{remark}
\newtheorem{rem}[thm]{Remark}

\numberwithin{equation}{section}
\newcommand{\R}{\mathbb R}
\newcommand{\N}{\mathbb N}
\newcommand{\C}{\mathcal{C}}
\newcommand{\ve}{\varepsilon}
\newcommand{\lam}{\lambda}

\begin{document}
\title{Homogenization of {F}u{\v{c}}{\'{\i}}k eigenvalues by optimal partition methods}

\author[A M Salort]{Ariel M. Salort }
\address{Departamento de Matem\'atica
 \hfill\break \indent FCEN - Universidad de Buenos Aires and
 \hfill\break \indent   IMAS - CONICET.
\hfill\break \indent Ciudad Universitaria, Pabell\'on I \hfill\break \indent   (1428)
Av. Cantilo s/n. \hfill\break \indent Buenos Aires, Argentina.}
\email[A.M. Salort]{asalort@dm.uba.ar}



 \begin{abstract}
Given a bounded domain  $\Omega$ in $\mathbb{R}^N$, $N\geq 1$ we study the asymptotic behavior as $\varepsilon \to 0$ of the eigencurves of 
$$
  -\Delta_p u_\varepsilon=\alpha_\varepsilon m(\tfrac{x}{\varepsilon})(u_\varepsilon^+ )^{p-1} - \beta_\varepsilon n(\tfrac{x}{\varepsilon})(u_\varepsilon^- )^{p-1} \quad \textrm{ in } \Omega
$$
with  Dirichlet  boundary conditions, where $m$ and $n$ are bounded periodic weights.  In this work we  obtain   accurate bounds of the convergence rates of these curves to some limit curves as $\ve \to 0$.
 \end{abstract}

\maketitle

\section{Introduction}

Given a bounded domain  $\Omega$  in $\R^N$, $N\geq 1$ we study the asymptotic behavior as $\varepsilon \to 0$ of the spectrum of the following asymmetric elliptic problem
\begin{align} \label{P1}
	\bigg\{
	\begin{array}{ll}
	-\Delta_p u_\varepsilon=\alpha_\varepsilon m_\varepsilon(u_\varepsilon^+ )^{p-1} - \beta_\varepsilon n_\varepsilon(u_\varepsilon^- )^{p-1} &\quad \textrm{ in } \Omega\\
	u_\ve=0 &\quad \mbox{on }\partial \Omega.
	\end{array}
\end{align}

Here, $\Delta_p u:=div(|\nabla u|^{p-2}\nabla u)$ denotes the $p-$Laplace operator  with $1<p<\infty$ and, as usual, $u^\pm:=\max\{\pm u, 0\}$. The parameters $\alpha_\varepsilon$ and $\beta_\varepsilon$ are real numbers  depending on $\varepsilon>0$.
Here the family of functions $m_\varepsilon$ and $n_\varepsilon$ are given in terms of $Q-$periodic functions, $Q$ being the unit cube in $R^N$, in the form $m_\ve(x)=m(x/\ve)$ and $n_\ve(x)=n(x/\ve)$. The functions $m$ and $n$  are assumed to be positive and uniformly bounded away from zero and infinity, that is, there are constants $\theta_-$, $\theta_+$ such that
\begin{equation} \label{cotas}
0<\theta_- \leq m(x), n(x) \leq \theta_+ < +\infty.
\end{equation}

It is well-known that as $\varepsilon \to 0$,
\begin{equation} \label{limi}
  m_\varepsilon(x)\rightharpoonup \bar m=\fint_Q m(x)\, dx, \quad
  n_\varepsilon(x)\rightharpoonup \bar n=\fint_Q n(x)\, dx  \quad\textrm{ weakly* in }L^\infty(\Omega).
\end{equation}

Problem \eqref{P1} was widely studied for a fixed value of $\ve>0$: see for instance Arias and Campos \cite{AR1}, Drabek \cite{DRA}, Reichel and Walter \cite{RE1}, Rynne and Walter \cite{RYN},  for positive weights; Alif and Gossez \cite{ALIF1}, Leadi and Marcos \cite{LIAM} for   indefinite weights.

For a fixed $\ve>0$, the Fu\v c\'\i k spectrum of \eqref{P1} is defined as the set
$$
	\Sigma_\ve=\Sigma_\ve(m_\ve,n_\ve):=\{(\alpha_\ve ,\beta_\ve) \in \R^2\colon (\ref{P1}) \textrm{ has a nontrivial solution} \}.
$$
Moreover, we say that a nontrivial function $u_\ve\in W^{1,p}_0(\Omega)$ is an eigenfunction of \eqref{P1} associated to $(\alpha_\ve,\beta_\ve)\in \R^+\times \R^+$ if it satisfies the weak formulation
\begin{equation} \label{debil}
	\int_\Omega |\nabla u|^{p-2} \nabla u \cdot \nabla v \, dx = \int_\Omega (\alpha m_\ve(x) (u_\ve^+)^{p-1} v - \beta_\ve n(x) (u_\ve^-)^{p-1} v)\,dx
\end{equation}
for all $v\in W^{1,p}_0(\Omega)$.

Observe that when both weights are the same, let us say, $r_\ve$, and both parameters are equal, let us say, $\lam_\ve$, equation (\ref{P1}) becomes the weighted $p-$laplacian eigenvalue problem with Dirichlet boundary conditions, i.e., 
\begin{align} \label{Plap}
\begin{cases}
-\Delta_p u_\ve= \lam_\ve r_\ve |u_\ve|^{p-2}u_\ve  &\quad \textrm{in } \Omega\\
u_\ve=0 &\quad \mbox{on }\partial \Omega.
\end{cases}
\end{align}

One immediately observe that $\Sigma_\ve$ contains the trivial lines $\lam_{1}(m_\ve)\times \R$ and $\R \times \lam_{1}(n_\ve)$, being $\lam_1(r_\ve)$ the first eigenvalue of \eqref{Plap}.
In contrast with the one-dimensional case, where a full description of the spectrum is obtained, when $N>1$ it is only known the existence
of a curve $\mathcal{C}_\ve$ beyond the trivial lines, see \cite{AR1,AR3}. Such curve can be written by considering its intersection with the line of slope $s\in \R^+$ passing through the origin in $\R^2$ as
\begin{align} \label{def_ce}
	\C_{\varepsilon}=\mathcal{C}_{\ve}(m_\varepsilon,n_\varepsilon):=\{(\alpha_\varepsilon(s),\beta_\varepsilon(s)), s\in\R^+\}.
\end{align}
The authors in \cite{AR3} deal with a variational characterization for $\alpha(s)$ and $\beta(s)$.
 
When $\varepsilon \to 0$   the following natural limit problem for (\ref{P1}) is obtained
\begin{equation} \label{pron1limite.gral}
	\bigg\{
	\begin{array}{ll}
	-\Delta_p u=\alpha_0 \bar m (u^+ )^{p-1} - \beta_0 \bar m (u^- )^{p-1} &\quad \textrm{ in } \Omega \\[0.05 cm]
	u=0 &\quad \textrm{ on } \partial \Omega
	\end{array}
\end{equation}
where $\bar m$ and $\bar n$ are given in (\ref{limi}), and whose  corresponding first nontrivial curve is denoted  by
$$
	\mathcal{C}_0=\mathcal{C}_0(m_0,n_0):=\{(\alpha_0(s),\beta_0(s)), s\in\R^+\}.
$$

In this context, in the previous work \cite{SA-fucik} it was stated the convergence of $\C_\ve$ to the limit curve $\C_0$ (even for non-periodic weights) in the sense that
$$
	\alpha_\varepsilon(s)\to \alpha_0(s) \quad \mbox{and} \quad \beta_\varepsilon(s)\to \beta_0(s)
$$ 
as $\ve \to 0$, for each fixed $s\in\R^+$. Moreover,  by using the variational characterization of $\C_\ve$ and $\C_0$ provided by \cite{AR3}, it was established the convergence rates of the curves:

\begin{thm}[Theorem 4.2, \cite{SA-fucik}]\label{teo_1.viejo}
Given $\ve>0$ and $s\in \R^+$, let $(\alpha_\ve(s),\beta_\ve(s))\in \C_\ve$ and $(\alpha_0(s),\beta_0(s))\in \C_0$. Then the following estimates hold
\begin{align} \label{estim1}
	|\alpha_\varepsilon(s) - \alpha_0(s)|\leq 
	 \begin{cases}
	 c\ve s &\quad s\geq 1\\
	 c\ve s^{-2} &\quad s< 1,
	 \end{cases}
	\quad
	|\beta_\varepsilon(s) - \beta_0(s)| \leq
	 \begin{cases}
	 c\ve s^2 &\quad s\geq 1\\
	 c\ve s^{-1} &\quad s< 1
	 \end{cases}
\end{align}
where $c$ is a computable constant independent on $\ve$ and $s$.
 
\end{thm}

Nevertheless, since  estimates \eqref{estim1} do not depend on $p$, we suspect that Theorem \ref{teo_1.viejo} does not turn to be  enough accurate. 

Our first aim in this paper is to refine \eqref{estim1} by using an alternative characterization of the curves. By following the arguments of  \cite{terra} it is possible to define $\mathcal{C}_\ve$ by minimizing the first eigenvalue of weighted $p-$Laplacian problems over all possible partition of the kind 
$$
	\mathcal{P}=\{\{\omega_+,\omega_-\}\subset \Omega \, : \, \omega_\pm \mbox{ is open and connected, }\omega_+\cap \omega_-=\emptyset\},
$$ 
see Theorem \ref{teocurva} in Section \ref{sec.fucik.rn} for the precise statement. Such optimal partition characterization   reduces our analysis to studying the  homogenization rates of the first eigenvalue of the weighted $p-$laplacian.

Our first result reads as follows.

\begin{thm}  \label{main1}
Given $\ve>0$ and $s\in \R^+$, let $(\alpha_\ve(s),\beta_\ve(s))\in \C_\ve$ and $(\alpha_0(s),\beta_0(s))\in \C_0$. Then the following estimates hold
\begin{align*}
	|\alpha_\varepsilon(s) - \alpha_0(s)|\leq 
	 \begin{cases}
	 C\ve s^\frac{1}{p} &\quad s\geq 1\\
	 C\ve s^{-1-\tfrac{1}{p}} &\quad s< 1
	 \end{cases}
	\qquad
	|\beta_\varepsilon(s) - \beta_0(s)| \leq
	 \begin{cases}
	 C\ve s^{1+\frac{1}{p}} &\quad s\geq 1\\
	 C\ve s^{-\frac{1}{p}} &\quad s< 1
	 \end{cases}
\end{align*}
where $C$ is a constant independent on $\ve$ and $s$. 
\end{thm}

\begin{rem}
	A careful computation allow us to compute explicitly the constant in Theorem \ref{main1} as 
	\begin{equation} \label{cte}
		C=\left(\frac{\theta_+}{\theta_-}\right)^{1+\frac{1}{p}}\mu_2(\Omega)^{1+\frac{1}{p}} \max\{C_m,C_n\}
	\end{equation}
	where $\mu_2$ is the second eigenvalue of the Dirichlet $p-$laplacian in $\Omega$ and
\begin{equation}  \label{ctee}
		C_r= p\frac{\sqrt{N}}{2} \|r-\bar r\|_{L^\infty(\R^N)} \theta_+ (\theta_-)^{-\frac{1}{p}-2}.
\end{equation}
 
\end{rem}

In the second part of the work we deal with the homogenization of the one-dimensional version of \eqref{P1}, i.e., 
	\begin{align} \label{P1.1D}
	\bigg\{
	\begin{array}{ll}
	-\Delta_p u_\varepsilon=\alpha_\varepsilon m_\varepsilon(u_\varepsilon^+ )^{p-1} - \beta_\varepsilon n_\varepsilon(u_\varepsilon^- )^{p-1} &\quad \textrm{ in } (a,b)\subset \R\\
	u_\ve(a)=u_\ve(b)=0.
	\end{array}
\end{align}

Problem \eqref{P1.1D} was introduced in the '70s by Dancer and {F}u\v c\'\i k (see \cite{DAN, FUCiK-libro}) for constant weights and a fixed value of $\ve>0$. These authors were interested in problems with jumping nonlinearities, and  obtained that the nontrivial solutions consist in a  family of hyperbolic-like curves.

The existence of similar curves in the spectrum was proved later for non-constant weights by Rynne in \cite{RYN}, together with several properties about simplicity of zeros. The asymptotic behavior of the curves was studied in \cite{PiS}. For sign-changing weights, similar results were obtained by Alif and Gossez, see \cite{ALIF1}.

The main advantage with regard to the higher dimensional case is the fact of knowing the structure of the whole spectrum and  precise information about the curves. By means of shooting arguments Rynne proved that the spectrum of \eqref{P1.1D} can be described as an union of curves
$$
	\Sigma_\ve(m_\ve,n_\ve):=  \bigcup_{k\in \N_0} \mathcal{C}_{k,\ve}.
$$
Here, $\C_{0,\ve}=\C_{0,\ve}^+ \cup \C_{0,\ve}^-$ are the trivial lines, which are given by
$\lam_{1}(m_\ve)\times \R$, and  $\R \times \lam_{1}(n_\ve)$,  respectively, being $\lam_1(r_\ve)$ the first eigenvalue of the Dirichlet $p-$laplacian with weight $r_\ve$. These curve are characterized for having eigenfunctions which do not change signs.
The remaining curves are made as the union $\mathcal{C}_{k,\ve}=\mathcal{C}_{k,\ve}^+ \cup \mathcal{C}_{k,\ve}^-$, where $\mathcal{C}_{k,\ve}^+$ (resp. $\mathcal{C}_{k,\ve}^-$) it is composed of pairs  whose corresponding eigenfunctions have $k$ internal zeros, and positive (resp. negative) slope at $x=a$. 

As $\ve\to 0$, the following natural limit problem for \eqref{P1.1D}  is obtained
\begin{align}  \label{P1.1D.lim}
\bigg\{
\begin{array}{ll}
-\Delta_p u=\alpha_0 \bar m(u^+ )^{p-1} - \beta_0 \bar n(u^- )^{p-1} &\quad \textrm{ in } (a,b)\\
u(a)=u(b)=0,
\end{array}
\end{align}
where  $\bar m$ and $\bar n$ are given in (\ref{limi}). Similarly, its corresponding spectrum is composed as the union
$$
	\Sigma_0(n_0,n_0):= \bigcup_{k\in \N_0} \mathcal{C}_{k,0} ,
$$
with curves $\mathcal{C}_{k,0} $ satisfying  analogous properties to $\mathcal{C}_{k,\ve} $.

In order to describe the curves in the spectrum of $\Sigma_\ve$ and $\Sigma_0$ we denote $(\alpha_{k,\ve}(s),\beta_{k,\ve}(s))$ and  $(\alpha_{k,0}(s),\beta_{k,0}(s))$ the intersection of the curves $\mathcal{C}_{k,\ve}$ and $\mathcal{C}_{k,0}$ with  the line of slope $s$ passing through the origin, respectively.

\medskip

Under these considerations, in \cite{FBPS-fucik} it was studied  the   behavior of the  eigencurves of \eqref{P1.1D} as $\ve$ approaches zero. It was proved that for each $k\in \N_0$, the curve $\C_{k,\ve}$ converges to  $\C_{k,0}$ in the sense that
\begin{equation} \label{convv}
	\alpha_{k,\ve}(s)\to \alpha_{k,0}(s) \quad \mbox{and} \quad \beta_{k,\varepsilon}(s)\to \beta_{k,0}(s)
\end{equation}
as $\ve \to 0$, for each fixed $s\in\R^+$.

Again, following  \cite{terra} it is possible to obtain a representation of the curves $\mathcal{C}_{k,\ve}$ and $\mathcal{C}_{k,0}$ by minimizing the first eigenvalue of weighted $p-$laplacian problems over all possible partition of the kind 
$$
	\mathcal{P}_{k+1}\, : \, \{a=t_0<t_1<\ldots<t_{k+1}=b\},
$$ 
see Theorem \ref{teo_part_1d} in Section \ref{sec.1d} for the precise statement. Such characterization allow us to prove the following result concerning to the convergence rates of \eqref{convv}:

\begin{thm} \label{teo-1d}
Given $\ve>0$ and $s\in \R^+$, let $(\alpha_\ve(s),\beta_\ve(s))\in \C_{k,\ve}$ and $(\alpha_0(s),\beta_0(s))\in \C_{k,0}$. Then the following estimates hold
\begin{align*}
	|\alpha_\varepsilon(s) - \alpha_0(s)|\leq 
	 \begin{cases}
	 C\ve k^{p+1}s^\frac{1}{p} &\, s\geq 1\\
	 C\ve k^{p+1} s^{-1-\tfrac{1}{p}} &\, s< 1,
	 \end{cases}
	\,\,
	|\beta_\varepsilon(s) - \beta_0(s)| \leq
	 \begin{cases}
	 C\ve k^{p+1} s^{1+\frac{1}{p}} &\, s\geq 1\\
	 C\ve k^{p+1} s^{-\frac{1}{p}} &\, s< 1
	 \end{cases}
\end{align*}
where $(\alpha_0(s),\beta_0(s))\in \mathcal{C}_{k,0}$ and
$$
	C=\left(\frac{\theta_+}{\theta_-}\right)^{1+\frac{1}{p}} \left(\frac{\pi_p}{b-a} \right)^{1+p} \max\{C_m,C_n\},
$$
being $C_m$ and $C_n$ given in  \eqref{ctee}.
\end{thm}

\medskip

Observe that when we specialize Theorem \ref{teo-1d} with both weight functions  being the same $1-$periodic function $r(\tfrac{x}{\ve})$, and both parameters being the same, that is, $\lam_\ve=\alpha_\ve=\beta_\ve$, it follows that   $s=1$, and we recover the homogenization rates for the eigenvalue convergence of 
\begin{equation} \label{er1}
	-\Delta_p u_\ve = \lam_\ve r_\ve |u_\ve|^{p-2}u_\ve \quad \mbox{in }(a,b), \qquad u_\ve(a)=u_\ve(b)=0,
\end{equation}
to the limit problem
\begin{equation} \label{er2}
-\Delta_p u = \lam_0 \bar r |u|^{p-2}u \quad \mbox{in }(a,b), \qquad u(a)=u(b)=0
\end{equation}
as $\ve\to 0$, which has been widely studied, see for instance \cite{FBPS1, FBPS2, FBPS3}. More precisely, in particular Theorem \ref{teo-1d} states that 
\begin{equation*}
	|\lam_{k}(r_\ve)-\lam_{k}(\bar r)|\leq ck^{p+1}\ve
\end{equation*}
where $\lam_{k}(r_\ve)$ is $k-$th eigenvalue of \eqref{er1}, $\lam_{k}(\bar r)$ is the $k-$th eigenvalue of \eqref{er2}, and  $c$  is a constant independent on $k$ and $\ve$.

\medskip 

The paper is organized as follows: in Section \ref{sec.fucik.rn} we state some properties concerning to the first nontrivial curve in the Fu{\v{c}}{\'{\i}}k spectrum for $N\geq 1$; in Section \ref{seccion.part} we deal with the proof of Theorem \ref{main1}; in Section \ref{sec.1d} we study the one-dimensional Fu{\v{c}}{\'{\i}}k eigencurves; finally in Section \ref{sect1d.p} we provide a proof for Theorem \ref{teo-1d}.

 \section{The Fu{\v{c}}{\'{\i}}k spectrum in $\R^N$} \label{sec.fucik.rn}

As we pointed in the introduction, given $\Omega\subset\R^N$, $N\geq 1$, and functions $m$ and $n$ satisfying \eqref{cotas}, the structure of the spectrum $\Sigma(m,n)$ of the following asymmetric equation
\begin{align} \label{P10}
\bigg\{
\begin{array}{ll}
-\Delta_p u=\alpha m(u^+ )^{p-1} - \beta n(u^- )^{p-1} &\quad \textrm{ in } \Omega\\
u=0 &\quad \mbox{on }\partial \Omega.
\end{array}
\end{align}
is not completely understood, even in the constant weight case. Immediately one can check that $\Sigma(m,n)$ contains the lines $\lam_1(m)\times \R$ and $\R \times \lam_1(n)$. Here, given a function $r$ satisfying \eqref{cotas}, $\lam_1(r)$ denotes the first eigenvalue of 
\begin{align} \label{p.lap}
\begin{cases}
	-\Delta_p u = \lam  r  |u|^{p-2}u  &\quad \textrm{in } \Omega\\
	u=0 &\quad \mbox{on }\partial \Omega.
\end{cases}
\end{align}
The first eigenvalue of \eqref{p.lap} can be written variationally  by minimizing the following quotient over all the non-zero functions belonging to $W^{1,p}_0(\Omega)$
\begin{equation} \label{variac}
	\lam_1(r)=\inf \frac{\int_\Omega |\nabla u|^p \, dx}{\int_\Omega r(x)|u|^p\, dx}.
\end{equation}

When $r\equiv 1$ we just write $\mu_1$ to denote \eqref{variac}. When it is precise to empathize the dependence on the domain we will write $\lam_k(r,\Omega)$ and $\mu_k(\Omega)$ to denote the $k-$th variational eigenvalue of \eqref{p.lap}.

 In \cite{AR1,AR3} it was shown the existence of a first variational nontrivial curve $\C_1(m,n)$ given by minimizing the Rayleigh quotient associated to \eqref{P10} along a family of sign-changing paths.
More precisely, the authors in \cite{AR3} proved that
$$
	C_1(m,n)=\{(\alpha(s),\beta(s)), s\in \R^+\}
$$
where $(\alpha(s),\beta(s))$ is the intersection between $\Sigma(m,n)$ and the line of slope $s$ passing through the origin in $\R^2$. Each component is given by
$$
	\alpha(s)=c(m,sn), \qquad \beta(s)=s\alpha(s)
$$
where
$$
	c(m,n)=\inf_{\gamma \in \Gamma} \max_{u\in \gamma[-1,1]} \frac{\int_\Omega |\nabla u|^p\, dx}{\int_\Omega (m (u^+)^p+n (u^-)^p)\, dx}
$$
and $\Gamma:=\{\gamma \in C([-1,1])\colon \gamma(-1)\geq 0 \mbox{ and } \gamma(1)\leq 0 \}$.

Later on, problem \eqref{P10} was considered for the case $p=2$ and constant weights in \cite{terra}. In that paper the authors, among other things, obtain an alternative representation of $\C_1$ in the framework of optimal partitions, see Theorem 1.2 in \cite{terra}. However, as the they notice (in Remark 2.2, \cite{terra}) the procedure leading to the proof of such result can be trivially adapted for any $p\geq 2$ and by considering weights. Therefore, the result corresponding to \eqref{P10} can be stated as follows.

\begin{thm} \label{teocurva}
The first nontrivial curve $\C_1(m,n)$
in the spectrum of \eqref{P10} can be written as 
$$\C(m,n)=\{(\alpha(s),\beta(s)), s\in \R^+\},\qquad \mbox{ with } \quad  	\alpha(s)=s^{-1}c(s), \quad \beta(s)=c(s)$$
where
\begin{equation} \label{el_c}
	c(s):=\inf_{(\omega_i)\in \mathcal{P}_2} \max\{s\lam_{1}(m,\omega_+),\lam_{1}(n,\omega_-)\}
\end{equation}
and
$$
	\mathcal{P}_2=\{(\omega_+,\omega_-\}\subset \Omega \, : \, \omega_i \mbox{ is open and connected, }\omega_+\cap \omega_-=\emptyset\}.
$$ 
Moreover, for every $s>0$ there exists $u\in W^{1,p}_0(\Omega)$ such that $(\{u^+>0\},\{u^->0\})$ achieves $c(s)$.
\end{thm}

In order to prove our main result we state some properties concerning to the curve $\C_1$. First, we establish bounds for points belonging to $\C_1$ in terms of the parameter $s$ and the auxiliary function $\gamma:\R^+ \to \R^+$ defined as
	\begin{equation} \label{gama}
		\gamma(s)=
		\bigg\{
		\begin{array}{ll}
		1  &\mbox{ if }\ s \geq 1 \\
		s^{-1}  &\mbox{ if }\ s < 1.
		\end{array}
	\end{equation}

\begin{lema}[Lemma 3.1, \cite{SA-fucik}] \label{lema.2}
	Given $s\in \R^+$, let $(\alpha(s),\beta(s))\in \mathcal C_1(m,n)$. Then
	$$
		\alpha(s) \leq \theta_-^{-1} \mu_2(\Omega) \gamma(s), \qquad \beta(s) \leq \theta_-^{-1} \mu_2(\Omega) s \gamma(s)
	$$
	where $\gamma$ is defined in \eqref{gama} and $\mu_2$ denotes the second eigenvalue of the $p-$laplacian  in $\Omega$   with Dirichlet boundary conditions.
\end{lema}

In the following lemma we consider the first eigenvalue of the $p-$laplacian on nodal domains of eigenfunctions corresponding to points belonging to $\C_1$.

\begin{lema} \label{lema.3}
	Given $s\in \R^+$, let $(\alpha(s),\beta(s))\in \C_1(m,n)$ and let $u$ be a corresponding eigenfunction. If we denote $\omega_\pm=supp(u^\pm)$, then
	$$
		\mu_1(\omega_+) \leq C\gamma(s), \qquad   \mu_1(\omega_-)\leq Cs\gamma(s)
	$$
	where $C=\frac{\theta_+}{\theta_-}\mu_2(\Omega)$ and $\gamma(s)$ is given in \eqref{gama}.
\end{lema}
\begin{proof}
	By taking $v=u^+$ in the weak formulation of \eqref{P10} we obtain that
	\begin{align} \label{ef1}
		  \int_{\omega_+} |\nabla u^+|^p \, dx &= \int_{\Omega} |\nabla u^+|^p \, dx  =\alpha \int_{\Omega} m |u^+|^p\, dx = \alpha\int_{\omega_+} m |u^+|^p\, dx,
	\end{align}
	from where it follows that $\alpha=\lam_1(m,\omega_+)$ and $u|_{\omega_+}\in W^{1,p}_0(\omega_+)$ is an eigenfunction associated  to $\lam_1(m,\omega_+)$.
	Since
	$$
		\frac{1}{\theta_+}\frac{\int_{\omega_+} |\nabla v|^p}{\int_{\omega^+} |v|^p}\leq  \frac{\int_{\omega_+} |\nabla v|^p}{\int_{\omega_+}m |v|^p}  \leq \frac{1}{\theta_-}\frac{\int_{\omega_+} |\nabla v|^p}{\int_{\omega_+} |v|^p}
	$$
	for all $v \in W^{1,p}_0(\omega_+)$, from \eqref{variac} it follows that
	$$
		\tfrac{1}{\theta_+}\mu_1(\omega_+ )  \leq \lam_1(m,\omega_+) \leq \tfrac{1}{\theta_-}\mu_1(\omega_+ ),
	$$
	and the desired inequality follows by using Lemma \ref{lema.2}.  Analogously,  by using $u^-$ as a test function in the weak formulation of \eqref{P10}  the another inequality is obtained.

\end{proof}

\section{proof of the results for $N\geq 1$}  \label{seccion.part}

Before proving our main result, we state an auxiliary results concerning to the homogenization of eigenvalues of the weighted $p$-laplacian.

Given a bounded domain $\Omega\subset \R^N$ and a function $r$ satisfying \eqref{cotas}, we denote $\lam_{1}(r_\ve)$ the first eigenvalue of
  \begin{align}  \label{Plap1.ve}
  	\begin{cases}
  	-\Delta_p u_\ve= \lam_\ve r(\tfrac{x}{\ve}) |u_\ve|^{p-2}u_\ve  &\quad \textrm{in } \Omega\\
  	u_\ve=0 &\quad \mbox{on }\partial \Omega.
  	\end{cases}
  \end{align}
When an explicit emphasis on the domain is required, we denote $\lam_{1}(r_\ve,\Omega)$ the first eigenvalue of \eqref{Plap1.ve}; additionally, when $r_\ve\equiv 1$ we  write $\mu_{1}(\Omega)$ instead of $\lam_{1}(1,\Omega)$.

As we pointed in the introduction, convergence rates in the homogenization of the {F}u\v c\'\i k spectrum are closely related with the convergence rates in the homogenization of eigenvalues of the $p-$laplacian. Given a $Q-$periodic function $r$ satisfying \eqref{cotas}, $Q$ being the unit cube in $\R^N$, as $\ve\to 0$ the following limit problem for \eqref{Plap1.ve} is obtained
\begin{align} \label{Plap1.lim}
	\begin{cases}
	-\Delta_p u_0= \lam_{0} \bar r |u_0|^{p-2}u_0  &\quad \textrm{in } \Omega\\
	u_0=0 &\quad \mbox{on }\partial \Omega,
	\end{cases}
\end{align}
where $\bar r$ is the average of $r$ over $Q$. The eigenvalue $\lam_1(r_\ve)$ converges to the first eigenvalue of \eqref{Plap1.lim}. Furthermore, the rate  of the convergence of $\lam_{1}(r_\ve)$  is stated in the following result.

\begin{thm}[Theorem 2.2, \cite{SA-fucik}]\label{teo_1}
	Given a $Q-$periodic function $r$ satisfying \eqref{cotas}, let us denote $\lam_{1}(r_\ve)$ and $\lam_{1}(\bar r)$ the first eigenvalue  of equations \eqref{Plap1.ve} and \eqref{Plap1.lim}, respectively. Then
	$$|\lam_{1}(r_\ve,\Omega)-\lam_{1}(\bar r,\Omega)|\leq C_r \mu_1(\Omega)^{\frac{1}{p}+1} \varepsilon$$
	with $C_r$ given by
	$$
		C_r= p\frac{\sqrt{N}}{2} \|r-\bar r\|_{L^\infty(\R^N)} \theta_+ (\theta_-)^{-\frac{1}{p}-2}.
	$$
\end{thm}

We are ready to prove our main result in this section.

\begin{proof}[Proof of Theorem \ref{main1}]
According to Theorem \ref{teocurva} the curve $\mathcal{C}_\ve$ associated to  \eqref{P1} is given by 
$$
	\mathcal{C}_\varepsilon:=\{(\alpha_\ve(s),\beta_\varepsilon(s)),\, s\in\R^+\} = \{(s^{-1}c_\ve(s) , c_\ve(s)),\, s\in \R^+\}
$$
where 
\begin{equation} \label{el_c_eps}
	c_\ve(s):=\inf_{(\xi_+,\xi_-)\in \mathcal{P}_2} \max\{s\lam_1(m_\ve,\xi_+),\lam_1(n_\ve,\xi_-)\}.
\end{equation}
In a similar way the limit curve  $\mathcal{C}_0$ associated to \eqref{pron1limite.gral} is given by 
$$
	\mathcal{C}_0:=\{(\alpha_0(s),\beta_0(s)),\, s\in\R^+\} = \{(s^{-1}c_0(s) , c_0(s)),\, s\in \R^+\}
$$
where 
\begin{equation} \label{el_c_0}
	c_0(s):=\inf_{(\xi_+,\xi_-) \in \mathcal{P}_2} \max\{s\lam_1(\bar m,\xi_+),\lam_1(\bar n,\xi_-)\}.
\end{equation}
Let $(\omega_+,\omega_-)\in \mathcal{P}_2$ be a partition such that
$$
	c_0(s)= \max\{s\lam_1(\bar m,\omega_+),\lam_1(\bar n,\omega_-)\}.
$$
By putting $(\omega_+,\omega_-)$ in \eqref{el_c_eps} it follows that
\begin{align} \label{eq.c.0}
	\begin{split}
		c_\ve(s)&\leq \max\{s\lam_1(m_\ve,\omega_+),\lam_1(n_\ve,\omega_-)\}.
	\end{split}
\end{align}
Now, Theorem \ref{teo_1} allows as to bound $\lam_1(m_\ve,\omega_+)$ and $\lam_1(n_\ve,\omega_-)$ in terms of $\lam_1(\bar m,\omega_+)$ and $\lam_1(\bar n,\omega_-)$, from where we bound \eqref{eq.c.0} as
\begin{align} \label{eq.c.1}
\begin{split}
		 \max\{&s\big(\lam_1(\bar m ,\omega_+)+ C_m \mu_1(\omega_+)^{\frac{1}{p}+1} \varepsilon	
		\big),\lam_1(\bar n,\omega_-) + C_n \mu_1(\omega_-)^{\frac{1}{p}+1} \varepsilon \}\leq \\
		&\leq  \max\{s\lam_1(\bar m ,\omega_+) ,\lam_1(\bar n,\omega_-)\}+C_1\ve  \max\{ s\mu_1(\omega_+)^{\frac{1}{p}+1},\mu_1(\omega_-)^{\frac{1}{p}+1}   \}   \\
		&=  c_0(s) +C_1\ve  \max\{ s\mu_1(\omega_+)^{\frac{1}{p}+1},\mu_1(\omega_-)^{\frac{1}{p}+1}   \}   		
\end{split}		
\end{align}
where $C_1=\max\{C_m,C_n\}$.

In the another hand, by using Lemma \ref{lema.3} we obtain that
\begin{align} \label{eq.c.11}
\begin{split}
	\max\{ &s\mu_1(\omega_+)^{ \frac{1}{p}+1},\mu_1(\omega_-)^{\frac{1}{p}+1} \}\leq \\
	& \leq 	
\left(\frac{\theta_+}{\theta_-}\right)^{1+\frac{1}{p}} \max\{s (\mu_2(\omega_+)\gamma(s))^{\frac{1}{p}+1},(s\mu_2(\omega_-)\gamma(s))^{\frac{1}{p}+1} \}\\	
	&\leq C_2 \max\{s \gamma(s)^{\frac{1}{p}+1},(s\gamma(s))^{\frac{1}{p}+1} \}\\
	&\leq C_2  \gamma(s)^{1+\frac{1}{p}} s \max\{1,s^\frac{1}{p}\},
\end{split}		
\end{align}
where $C_2=\left(\frac{\theta_+}{\theta_-}\mu_2(\Omega)\right)^{1+\frac{1}{p}}$.

Collecting \eqref{eq.c.0}--\eqref{eq.c.11} we obtain that
\begin{align} \label{eq.c.2}
	\begin{split}
		c_\ve(s)\leq c_0(s)+ C \ve \gamma(s)^{1+\frac{1}{p}}s\max\{1,s^\frac{1}{p}\},
	\end{split}
\end{align}
where $C=C_1C_2$.
Interchanging the roles of $c_\ve(s)$ and $c_0(s)$ we similarly obtain that
\begin{align} \label{eq.c.21}
	c_0(s)\leq c_\ve(s)+ C \ve \gamma(s)^{1+\frac{1}{p}}s\max\{1,s^\frac{1}{p}\},
\end{align}
where $C$ is the same constant that in \eqref{eq.c.2}.

Mixing up \eqref{eq.c.2} and \eqref{eq.c.21} it follows that
$$
	|c_\ve(s)-c_0(s)|\leq    C \ve \gamma(s)^{1+\frac{1}{p}}s\max\{1,s^\frac{1}{p}\}.
$$

Now, from Theorem \ref{teocurva} we get 
\begin{align*}
	|\beta_\ve(s)-\beta_0(s)|&=|c_\ve(s)-c_0(s)|\leq
	C \ve \gamma(s)^{1+\frac{1}{p}}s\max\{1,s^\frac{1}{p}\},\\
	|\alpha_\ve(s)-\alpha_0(s)|&=s^{-1}|c_\ve(s)-c_0(s)|\leq
	C \ve \gamma(s)^{1+\frac{1}{p}}\max\{1,s^\frac{1}{p}\}
\end{align*}
as it was required.
\end{proof}

\section{The one-dimensional Fu{\v{c}}{\'{\i}}k problem} \label{sec.1d}
In this section we state some properties related with the following one-dimensional asymmetric equation in $\Omega=(a,b)$
\begin{align} \label{P10.1d}
	\bigg\{
	\begin{array}{ll}
	-\Delta_p u=\alpha m(u^+ )^{p-1} - \beta n(u^- )^{p-1} &\quad \textrm{ in } \Omega\\
	u(a)=u(b)=0.
	\end{array}
\end{align}
As it was pointed in the introduction,  Rynne \cite{RYN} shown that its spectrum is given by
$$
	\Sigma(m,n):= \bigcup_{k\in \N} \mathcal{C}_{k},
$$
where the curves $\C_k=\C_k^+\cup \C_k^-$, $k\in \N_0$ are composed of pairs $(\alpha,\beta)\in\R^2$ whose corresponding eigenfunctions have $k$ internal zeros and positive (resp. negative) slote at $x=a$. In particular, $\C_0^+=\lambda_1(m)\times \R$ and $\C_0^-=\R\times \lam_1(n)$ have eigenfunctions which do not change signs in $\Omega$, being $\lam_1(r)$ the first eigenvalue of
\begin{align} \label{lap.1d}
	\bigg\{
	\begin{array}{ll}
	-\Delta_p u=\lam r|u|^{p-2}u &\quad \textrm{ in } \Omega\\
	u(a)=u(b)=0.
	\end{array}.
\end{align}
For simplicity, when the $r=1$ we denote $\mu_k$ the $k-$th eigenvalue of \eqref{lap.1d}.

Sometimes, in order to empathize the dependence on the domain we write $\lam_k(r,\Omega)$ and $\mu_k(\Omega)$ to denote the $k-$th eigenvalues of \eqref{lap.1d}.

Observe that, in contrast with the higher dimensional case,  eigenvalues of the Dirichlet $p-$laplacian can be explicitly computed as 
$$
	\mu_k(I)=\pi_p^p k^p |\Omega|^{-p}
 $$
where $\pi_p=2(p-1)^{1/p}\int_0^1(1-s^p)^{-1/p} \, ds$, see \cite{dPDM}.

Moreover, the sequence of variational eigenvalues of \eqref{lap.1d} can be described as
\begin{equation}\label{variac.1d}
	\lam_k = \inf_{U \in T_k}  \sup_{u \in C} \frac{\int_a^b |u'|^p\, dx}{\int_a^b r(x)|u|^p\, dx}
\end{equation}
where
\begin{align*}
	T_k & = \{ U \subset W^{1,p}_0(\Omega) \  : \  U \  \mbox{ is compact,} \ U=-U, \
	\gamma(U) \ge k \},
\end{align*}
and $\gamma$ is the Krasnoselskii genus, see \cite{GAP} for details.

From \eqref{variac.1d} it follows that
\begin{equation} \label{eig.comp}
	\theta_+^{-1}\mu_k \leq \lam_k \leq \theta_-^{-1} \mu_k
\end{equation}
for any $k\geq 1$.

The paper  \cite{terra} characterizes the curves of $\Sigma(m,n)$ in terms of the first eigenvalue of weighted $p-$laplacian problems (see Theorem 1.3 and Remark 2.2). The description of the curves is made as follows. A couple $(\alpha,\beta)$ belonging to  $\mathcal{C}_{1}$ has eigenfunctions with an internal zero, i.e., it has two nodal domains. Such couple can be written as $(s^{-1}c_{2}(s),c_{2}(s))$, where
$$
	c_{2}(s)=\inf \max\{s \lam_1(m,I_1),\lam_2(n,I_2)\}
$$
and $s$ is the slope of the line $\ell_s$ passing through the origin such that $(\alpha,\beta)=\mathcal{C}_{1} \cap \ell_s$. The infumum is taken over all the partitions $\mathcal{P}_2$ of $\Omega$ such that $a=t_0<t_1<t_2=b$, and $I_1=t_1-t_0$, $I_2=t_2-t_1$.

Now, a couple belonging to $\mathcal{C}_{2}$  has associated eigenfunctions with three nodal domains. Such pair can be characterized as $(s^{-1}c_{3}(s),c_{3}(s))$, where
$$
	c_{3}(s)=\inf \max\{s \lam_1(m,I_1),\lam_2(n,I_2), s\lam_1(m,I_3)\}
$$
and $s$ is the slope of the line $\ell_s$ passing through the origin such that $(\alpha,\beta)=\mathcal{C}_{2} \cap \ell_s$. Here the infumum is taken over all the partition $\mathcal{P}_3$ of $\Omega$ such that $a=t_0<t_1<t_2<t_3=b$, with $I_1=t_1-t_0$, $I_2=t_2-t_1$ and  $I_3=t_3-t_2$.

In order to state the general case we introduce the following notation: for $k\geq 0$ we denote 
  $$\mathcal{P}_{k+1}\, : \{\,a=t_0<t_1<\ldots<t_{k+1}=b\}$$ 
a partition of $\Omega=(a,b)$, and we write $I_{i+1}=t_{i+1}-t_i$ for $0\leq i \leq k$.

\medskip
Although the result in \cite{terra} was proved for the case $p=2$ of \eqref{P10.1d} and with constant weights, as the authors comment, by mixing Theorem 1.3 and Remark 2.2 from \cite{terra} it is straightforward to obtain the following result concerning to the spectrum of the weighted equation \eqref{P10.1d} for any $p>2$.

\begin{thm} \label{teo_part_1d}
Given $k\geq 1$ let us define
\begin{align} \label{elcn1}
\begin{split}
	c_{k+1}^+(s)=\inf_{\mathcal{P}_{k+1}} \max_{0\leq i\leq k} \{s \lam_1(m,I_{2i+1}), \lam_1(n,I_{2i+2})\},\\
	c_{k+1}^-(s)=\inf_{\mathcal{P}_{k+1}} \max_{0\leq i\leq k} \{s \lam_1(m,I_{2i+2}), \lam_1(n,I_{2i+1})\}
\end{split}
\end{align}
for all $s>0$. Then the pair $(s^{-1}c_{k+1}^\pm(s),c_{k+1}^\pm,(s))$ belongs to a curve $\mathcal{C}_{k}^\pm$.

Moreover,  the infima above are attained for suitable optimal partitions $P^\pm\in\mathcal{P}_{k+1}$. Furthermore, there are  eigenfunctions $u^\pm\in W^{1,p}_0(\Omega)$ of \eqref{P10.1d} associated to $(\alpha=s^{-1}c_{k+1}^\pm(s),\beta=c_{k+1}^\pm,(s))$ whose nodal domains are given by $P^\pm$.
\end{thm}

From the definition of $c_{k+1}(s)$ it is easy to check that $s_2>s_1$ implies  $c_{k+1}(s_2)>c_{k+1}(s_1)$. Moreover, it can be proved that $s_2^{-1}c_{k+1}(s_2)<s_1^{-1}c_{k+1}(s_1)$, from where the monotonicity of  $\mathcal{C}_{k+1}$ follows:

\begin{lema} [Theorem 21, \cite{RYN}]
The curve $\mathcal{C}_{k+1}$ is decreasing in the sense that if the points $(\alpha(s_1),\beta(s_1))$ and $(\alpha(s_2),\beta(s_2))$ belong to $\mathcal{C}_{k+1}$ then 
$$
	\alpha(s_1)> \alpha(s_2) \qquad \mbox{ and } \qquad \beta(s_2)>\beta(s_1)
$$
whenever $s_2>s_1$.
\end{lema}

The following inequality relates $c_{k}^\pm (1)$ with the $k-$th eigenvalue of the Dirichlet $p-$laplacian.

\begin{lema} \label{lema.4.3}
	Let $k\geq 1$ and  $c_{k+1}^\pm(\cdot )$ given in \eqref{elcn1}. It holds that
	$$
		c_{k+1}^\pm (1) \leq    \theta_-^{-1} \mu_{k+1}(\Omega).
	$$
\end{lema}
\begin{proof}
	By using \eqref{cotas} and \eqref{eig.comp} we have that
	\begin{align} \label{eg1}
	\begin{split}
		c^+_{k+1}(1)&\leq \inf_{\mathcal{P}_{k+1}} \max_i \{ \lam_1(\theta_-,I_{2i+1}), \lam_1(\theta_-,I_{2i+2}) \}\\
		&\leq \theta_-^{-1}\inf_{\mathcal{P}_{k+1}} \max_i \{ \mu_1(I_{i}) \}
		\end{split}
	\end{align}
	In particular, if we take an uniform partition of $\Omega$, i.e., $|I_i|=|\Omega|/(k+1)$, it follows that $\mu_1(I_i)= \pi_p^p|I_i|^{-p}=\mu_{k+1}(\Omega)$ for each $0\leq i \leq k$ and  the result follows.
\end{proof}

 As a consequence of Lemma \ref{lema.4.3}, we obtain upper bounds for $\alpha(s)$ and $\beta(s)$. The following result is a one-dimensional version of Lemma \ref{lema.2} for every curve in the spectrum of \eqref{P10.1d}.
\begin{lema} \label{lema.21d}
	Let   $(\alpha(s),\beta(s))\in \mathcal C_k(m,n)$. For each $s>0$ it holds that
	$$
		\alpha(s) \leq \theta_-^{-1} \mu_{k+1}(\Omega) \gamma(s), \qquad \beta(s) \leq \theta_-^{-1} \mu_{k+1}(\Omega) s \gamma(s)
	$$
	with $\gamma$ defined by
	\begin{equation} \label{gama.1}
		\gamma(s)=
		\bigg\{
		\begin{array}{ll}
		1  &\mbox{ if }\ s \geq 1 \\[0.05 cm]
		s^{-1}  &\mbox{ if }\ s \leq 1.
		\end{array}
	\end{equation}
\end{lema}

\begin{proof}
	Let $s>0$ and $(\alpha(s),\beta(s))\in \mathcal{C}_{k}$. From Theorem \ref{teo_part_1d} we can write $\alpha(s)=s^{-1}c_{k+1}(s)$ and $\beta(s)=c_{k+1}(s)$ (here $c_k$ denotes any of $c_k^\pm$). We empathize that $\mathcal{C}_{k}(m,n)$ is an decreasing curve.

	When $s\geq 1$, by using Lemma \ref{lema.4.3} we can bound
	\begin{align} \label{rell1}
		\alpha(s)\leq \alpha(1)=c_{k+1}(1) \leq \theta_-^{-p} \mu_{k+1}(\Omega).
	\end{align}
	When $s\leq 1$ we have that $\beta(s)\leq \beta(1)$, from where $
		s^{-1}\beta(s)\leq s^{-1}\beta(1)$.
	Since $\beta(s)=s\alpha(s)$,  we conclude that
	\begin{align} \label{rell3}
		\alpha(s)&=s^{-1}\beta(s) \leq s^{-1}\alpha(1) =s^{-1} c_{k+1}(1) \leq s^{-1} \theta_-^{-p}\mu_{k+1}(\Omega).
	\end{align}
	By using \eqref{rell1} and \eqref{rell3} together with the relation $\beta=s\alpha$ the conclusion of the lemma follows.
\end{proof}

Finally, the following lemma allow us to estimate eigenvalues of the $p-$laplacian on nodal domains corresponding to eigenfunctions of \eqref{P10.1d}.
\begin{lema} \label{1d.lema.3}
	Let $(\alpha(s),\beta(s))\in \mathcal{C}_k(m,n)$ with associated eigenfunction $u$. Let $I_+$ (resp. $I_-$) be a nodal domain of $u$ in which $u>0$ (resp. $u<0$).
	 Then
	$$
		\mu_1(I_{+})\leq C\gamma(s) , \qquad \mu_1(I_{-})\leq C s\gamma(s)
	$$
	where $C=\frac{\theta_+}{\theta_-}\mu_{k+1}(\Omega)$ and $\gamma(s)$ is given in \eqref{gama.1}.
\end{lema}

\begin{proof}
	By arguing in the same way that in the proof of Lemma \ref{lema.3} it is obtained that
	$$
		\mu_1(I_+) \leq \theta_+    \alpha(s), \qquad
		\mu_1(I_-) \leq \theta_+  \beta(s).
	$$
	The  result now follows by applying Lemma \ref{lema.21d}.
\end{proof}

\section{proof of the result in the one-dimensional case} \label{sect1d.p}

Aimed at proving our main result for one-dimensional Fu{\v{c}}{\'{\i}}k spectrum, first we introduce the following notation we will use along this section. As we have pointed in the introduction, given the bounded interval $\Omega=(a,b)\subset \R$ and a function $r$ satisfying \eqref{cotas}, we denote $\lam_{k}(r_\ve,\Omega)$ the $k-$th eigenvalue of 
\begin{align} \label{Plap.sec}
	\begin{cases}
	-\Delta_p u_\ve= \lam_\ve r_\ve |u_\ve|^{p-2}u_\ve  &\quad \textrm{in } \Omega\\
	u_\ve(a)=u_\ve(b)=0.
	\end{cases}
\end{align}
In the case in which $r\equiv 1$ we just put $\mu_k(\Omega)$.

 Observe that, since  $\C_{k,\ve} \to \C_{k,0}$ (see Theorem 1, \cite{FBPS-fucik}) in the sense that
	\begin{equation} \label{ccon}
		\alpha_{k,\ve}(s)\to \alpha_{k,0}(s) \quad \mbox{and} \quad \beta_{k,\varepsilon}(s)\to \beta_{k,0}(s)	
	\end{equation}
	where, for $s\in \R^+$, $$
		(\alpha_\ve(s),\beta_\ve(s))\in \mathcal{C}_{k\,\ve}(m_\ve,n_\ve) \quad \mbox{ and } \quad  (\alpha_0(s),\beta_0(s))\in \mathcal{C}_{k,0}(\bar m,\bar n)
	$$
	 eigenfunctions corresponding to $(\alpha_0(s),\beta_0(s))$ have  exactly $k$ nodal domains on $\Omega$.

\medskip 
	
	With the previous remarks and lemmas stated in Section \ref{sec.1d} we are ready to prove the rates of the convergences \eqref{ccon}.

\begin{proof}[Proof of Theorem \ref{teo-1d}]
We consider the curve $\mathcal{C}_{k,\ve}^+$. An eigenfunction corresponding to a pair over this curve has  positive slope at $x=a$, therefore it is positive over odd nodal domains and negative over even nodal domains. The treatment for $\mathcal{C}_{k,\ve}^-$ is analogous.

Let $s>0$. According to Theorem \ref{teo_part_1d} a pair $(\alpha_\ve(s),\beta_\varepsilon(s))\in \mathcal{C}_{k,\ve}^+$ can be written as
$$
	(\alpha_\ve(s),\beta_\varepsilon(s)) = (s^{-1}c_{k+1,\ve}(s) , c_{k+1,\ve}(s))
$$
where 
\begin{equation} \label{1d.el_c_eps}
	c_{k+1,\ve}(s)=\inf_{\mathcal{P}_{k+1}} \max_i \{s \lam_{1,\ve}(m_\ve,I_{2i+1}), \lam_{1,\ve}(n_\ve,I_{2i+2})\}
\end{equation}
and in a similar way, the limit pair $(\alpha_0(s),\beta_0(s))$ belonging to the limit curve  $\mathcal{C}_{k,0}^+$ can be written as
$$
	(\alpha_0(s),\beta_0(s)) = (s^{-1}c_{k+1,0}(s) , c_{k+1,0}(s))
$$
where 
\begin{equation} \label{1d.el_c_0}
	c_{k+1,0}(s)=\inf_{\mathcal{P}_{k+1}} \max_i \{s \lam_1(\bar m,I_{2i+1}), \lam_1(\bar n,I_{2i+2})\}.
\end{equation}

Let $P_{k+1}\in \mathcal{P}_{k+1}$ a partition where the infimum is  attained in \eqref{1d.el_c_0}. By considering $P_{k+1}$ in the expression \eqref{1d.el_c_eps} we get
\begin{align} \label{1d.eq.c.0}
	\begin{split}
		c_{k+1,\ve}(s)&\leq  \max_i \{s \lam_{1}(m_\ve,I_{2i+1}), \lam_{1}(n_\ve,I_{2i+2})\}.
	\end{split}
\end{align}
Now, using Theorem \ref{teo_1} we can bound $\lam_{1}(m_\ve,I_{2i+1})$ and $\lam_{1}(n_\ve,I_{2i+2})$ in term of $\lam_1(\bar m,I_{2i+1})$ and $\lam_1(\bar n,I_{2i+2})$, from where we find an upper bound of \eqref{1d.eq.c.0} as:
\begin{align} \label{1d.eq.c.1}
\begin{split}
		 &\max_i\{s\big(\lam_1(\bar m ,I_{2i+1})+ C_m \mu_1(I_{2i+1})^{\frac{1}{p}+1} \varepsilon	
		\big),\lam_1(\bar n,I_{2i+2}) + C_n \mu_1(I_{2i+2})^{\frac{1}{p}+1} \varepsilon \}\\
		&\leq  \max\{s\lam_1(\bar m ,I_{2i+1}) ,\lam_1(\bar n,I_{2i+2})\}+C\ve  \max\{ s\mu_1(I_{2i+1})^{\frac{1}{p}+1},\mu_1(I_{2i+2})^{\frac{1}{p}+1}   \}   \\
		&=  c_{k+1,0}(s) +C_1\ve  \max\{ s\mu_1(I_{2i+1})^{\frac{1}{p}+1},\mu_1(I_{2i+2})^{\frac{1}{p}+1}   \}   		
\end{split}		
\end{align}
where $C_1=\max\{C_m,C_n\}$.

On the other hand, by using Lemma \ref{1d.lema.3} we obtain that
\begin{align} \label{1d.eq.c.11}
\begin{split}
	\max\{s\mu_1&(I_{2i+1})^{\frac{1}{p}+1},\mu_1(I_{2i+2})^{\frac{1}{p}+1} \} \\
	 &\leq 	C_2 \max\{s (\mu_{k+1}(I)\gamma(s))^{\frac{1}{p}+1},(s\mu_{k+1}(I)\gamma(s))^{\frac{1}{p}+1} \}\\
	&\leq C_2  \mu_{k+1}(I)^{1+\frac{1}{p}}\gamma(s)^{1+\frac{1}{p}} s \max\{1,s^\frac{1}{p}\}.
\end{split}		
\end{align}
where $C_2=\left(\frac{\theta_+}{\theta_-}\right)^{\frac{1}{p}+1}$.

Collecting \eqref{1d.eq.c.0}--\eqref{1d.eq.c.11} we obtain that
\begin{align} \label{1d.eq.c.2}
	\begin{split}
		c_{k+1,\ve}(s)\leq c_{k+1,0}(s)+ C_1C_2 \ve \mu_{k+1}(I)^{1+\frac{1}{p}} \gamma(s)^{1+\frac{1}{p}}s\max\{1,s^\frac{1}{p}\}.
	\end{split}
\end{align}

Interchanging the roles of $c_{k+1,\ve}(s)$ and $c_{k+1,0}(s)$ we similarly obtain that
\begin{align} \label{1d.eq.c.21}
	c_{k+1,0}(s)\leq c_{k+1,\ve}(s)+ C_1C_2 \ve \mu_{k+1}(I)^{1+\frac{1}{p}} \gamma(s)^{1+\frac{1}{p}}s\max\{1,s^\frac{1}{p}\}.
\end{align}
Mixing up \eqref{1d.eq.c.2} and \eqref{1d.eq.c.21} it follows that
$$
	|c_\ve(s)-c_0(s)|\leq   C_1C_2 \ve \mu_{k+1}(I)^{1+\frac{1}{p}} \gamma(s)^{1+\frac{1}{p}}s\max\{1,s^\frac{1}{p}\}.
$$

Finally, since $\mu_k(I)=k^p \pi_p^p |I|^{-p}$, we get 
\begin{align*}
	|\beta_\ve(s)-\beta_0(s)|&=|c_{k+1,\ve}(s)-c_{k+1,0}(s)|\leq
	C \ve (k+1)^{p+1} \gamma(s)^{1+\frac{1}{p}}s\max\{1,s^\frac{1}{p}\},\\
	|\alpha_\ve(s)-\alpha_0(s)|&=s^{-1}|c_{k+1,\ve}(s)-c_{k+1,0}(s)|\leq
	C \ve (k+1)^{p+1} \gamma(s)^{1+\frac{1}{p}}\max\{1,s^\frac{1}{p}\}
\end{align*}
where $C=C_1C_2 \big(\frac{\pi_p}{b-a} \big)^{1+p}$, and the result is proved.

\end{proof}

\section{Acknowledgements} This paper was mostly written during a visit at Universit\'a degli Studi di Torino. The author wishes to thank to Prof. Susanna Terracini for her useful discussions about this topic and help along the stay.

\bibliographystyle{amsplain}
\bibliography{Biblio}

\end{document}